\documentclass[11pt, a4paper]{amsart}

\usepackage[colorinlistoftodos]{todonotes}

\usepackage[utf8]{inputenc}
\usepackage{lmodern}
\usepackage{amssymb,amsfonts,amsthm,amsmath,bbm,mathrsfs,aliascnt,mathtools}
\usepackage[english]{babel}
\usepackage[colorlinks]{hyperref}

\usepackage{ulem}
\usepackage{nicefrac}

\usepackage{tikz} 
\usepackage{subcaption}

\usepackage{graphicx} %
\usepackage{sidecap, caption}
\usepackage{wrapfig}

\usepackage{multicol}
\usepackage{lipsum}
\usepackage{array}

\theoremstyle{plain}
\newtheorem{thm}{Theorem}[section]
\newtheorem{cor}[thm]{Corollary}
\newtheorem{lem}[thm]{Lemma}
\newtheorem{prop}[thm]{Proposition}

\newtheorem{innercustomthm}{Proposition}
\newenvironment{myprop}[1]
{\renewcommand\theinnercustomthm{#1}\innercustomthm}
{\endinnercustomthm}

\theoremstyle{definition}
\newtheorem{definition}[thm]{Definition}

\newtheorem{ex}[thm]{Example}

\theoremstyle{remark}
\newtheorem{remark}[thm]{Remark}
\newtheorem{example}[thm]{Example}

\DeclareMathOperator{\supp}{supp}

\newcommand{\g}{\mathbf{g}}

\newcommand{\R}{\mathbb R}

\newcommand{\N}{\mathbb N}

\newcommand{\DD}{\mathcal D}

\newcommand{\II}{\mathcal I}
\newcommand{\JJ}{\mathcal J}

\newcommand{\MM}{\mathcal M}

\def\d{{\rm d}}

\newcommand{\lev}{\text{lev}}
\newcommand{\leb}{\text{Leb}}

\begin{document}

\title[Unique ergodicity...]{Unique ergodicity for random noninvertible maps on an interval}

\author{Sara Brofferio}
\address{
Univ Paris Est Creteil, Univ Gustave Eiffel, CNRS, LAMA UMR8050, F-94010 Creteil, France
(SB)}
\email{sara.brofferio@u-pec.fr}
\author{Hanna Oppelmayer}
\address{Universit\"at Innsbruck, Technikerstrasse 13, A--6020 Insbruck, Austria (HO)}
\email{Hanna.Oppelmayer@uibk.ac.at}
\author{Tomasz Szarek}
\address{Institut of Mathematics Polish Academy of Sciences, Abrahama 18, Sopot, Poland (TS)}
\email{tszarek@impan.pl}

\thanks{T.S. was supported by the Polish NCN Grant 2022/45/B/ST1/00135.\\ H.O. was partially supported by Early Stage Funding- Vicerector for research, Universität Innsbruck}

\date\today

\maketitle

\begin{abstract}  
In this short note, we investigate non-invertible stochastic dynamical systems on the unit interval $[0,1]$. We  provide a handy condition for  unique ergodicity for systems that are injective in mean.  
 On the other hand, we give concrete examples where unique ergodicity fails.
 \end{abstract}
 
\tableofcontents

\section{Introduction}
Ergodicity is the key concept in the theory of dynamical systems. It comes from statistical physics but also captures some nice properties of stochastic processes.
This note is concerned with the ergodic properties of Markov chains corresponding to random iteration of maps on the interval $[0, 1]$. 

Let $\MM$ be a set of piecewise monotone (with finitely many pieces) continuous functions on $[0, 1]$.
Let $\mu$ be a Borel probability measure on the space $C([0,1])$ (equipped with the topology of uniform convergence) such that  $\mu(\MM)=1$. The probability $\mu$ induces by recursion a Markov chain on $[0, 1]$, called a {\it stochastic dynamical system} (SDS),  and given   by the formula:
\begin{equation*}
X_n^x:= \g_n(X_{n-1}^x)\quad \mbox{for $n\ge 1$} \mbox{ and $X_0^x=x$.}
\end{equation*}
Here $(\g_n)_{n\in\N}$  is a sequence of independent  $\MM$--valued random variables with distribution $\mu$.
The transition probability of this Markov chain is given by the formula:
$$
\pi(x, A)=\int_{\MM} {\bf 1}_A(g(x))\, \d\mu(g)\qquad \mbox{ for $x\in\R$ and $A\in {\mathcal B}([0, 1])$.}
$$
A Borel probability measure $\nu$ on $[0,1]$ is called \textit{$\mu$--invariant} (also known as \textit{$\mu$--stationary}) if 
$$
\nu(A)=\int_{\MM} \nu(g^{-1}A)\, \d\mu(g), \qquad \text{ for every $A\in\mathcal{B}([0,1])$}.
$$
Observe that due to the compactness of $[0,1]$ and the continuity of the functions in $\MM$ there always exists at least one $\mu$--invariant probability measure on the closed interval $[0,1]$.
Our aim is to formulate sufficient conditions 
for the existence of a unique $\mu$--invariant  probability measure. We will also provide examples where uniqueness fails.

The problem of unique ergodicity for random dynamical systems has been intensively studied recently but mainly in the case when $\mu$ is supported on some subgroup of homeomorphisms (see \cite{AM, CS, DKN, GH, M, N}). The paper by A. Homburg {\it et al.} (see \cite{HKRVZ}) is in fact an exception. Namely, the authors prove unique ergodicity for some systems including logistic maps using the criterion for the existence of an invariant measure absolutely continuous with respect to Lebesgue measure derived
for unimodal maps with negative Schwarzian derivative (see \cite{NvS}).  Further, under the assumption of some average contractivity, Kloeckner \cite{K} proved unique ergodicity for quite general iterated function systems defined on complete metric spaces. Similar results on this topic were obtained by Czapla in \cite{C}.

The paper is aimed at proving unique ergodicity for SDSs that are generated by non--injective maps, but such that a generic point $x\in [0,1]$ has, on average, not more than one preimage. We call this property \textit{$\mu$-injectivity}. For such systems we can generalize some techniques developed for partially hyperbolic diffeomorphisms by Avila and Viana (see \cite{AV}) and extended to random homeomorphisms on the circle by Malicet (see \cite{M}).

\begin{definition} Let $\mu$ be a probability measure on  $C([0,1])$ with $\supp(\mu)=\MM$. We say that a Borel probability measure $\nu$ on $[0,1]$ is \textit{$\mu$--injective in $x\in[0,1]$} if 
$$\int_{\MM} n(g,x)\, d\mu(g)\leq 1,$$
where $n(g,x)$ denotes the cardinality of the preimage $g^{-1}(\{x\})$. 
\end{definition}
Furthermore, we say that a stochastic system $(\MM,\mu)$  \textit{contracts a neighbourhood of $x_0\in[0,1]$ } if there exists $\varepsilon>0$ such that
	\begin{align*}
	\mu(\{g\in\MM\, : \, 
  g(x_0)=x_0, \,  |g(x)-g(x_0)|<|x-x_0|, \, \forall x\in B(x_0,\epsilon)\})>0.
\end{align*}
where $B(x_0,\epsilon):= [x_0-\epsilon, x_0+\epsilon]\cap[0,1]$.

The two above conditions imply uniqueness as follows.
\begin{thm}\label{intro thm: mu inj implies pos h}
	Let $(\MM,\mu)$ be a stochastic dynamical system $\mu$--injective in all but countably many $x\in[0,1]$ and assume  that it contracts a neighbourhood of $x_0\in[0,1]$.	Let $\nu$ be an atomless, ergodic  $\mu$-invariant probability measure with $x_0\in\supp(\nu)$. Then any other $\mu$-invariant, ergodic Borel probability measure whose support contains $\supp(\nu)$ coincides with $\nu$. 
\end{thm} While the above theorem does not concern the existence of such measures, we provide conditions for the existence of a $\mu$-invariant probability measure on $(0,1)$ in Corollary \ref{cor-pract-cond} in Section~\ref{sec: 4}.\\

The paper is organized as follows.
We start by showing in Section~\ref{sec: 2} that for stochastic dynamical systems defined by piecewise monotone functions, it is possible to extend the notion of Furstenberg entropy $h_\mu(\nu)$ of an invariant probability $\nu$ based on Radon-Nikodym derivatives and a pushforward functional $g^{-1}\nu$. As in the invertible case, if the measure $\nu$ is atomless and the entropy is strictly positive, then the SDS is locally contracting (Proposition \ref{contractivity}) and ergodic measures are explicitely determined by their support (Theorem \ref{thm: pos h}). 
The proof of the contractivity (Proposition \ref{contractivity}) is an adaption of a proof by Malicet \cite{M} and Avila-Viana \cite{AV}, provided in the appendix. To adapt the proof we use a result about generalized Radon-Nikodym derivatives (Proposition~\ref{prop-RN-deri-NC-int}), which we also prove in the appendix.

In Section~\ref{sec: 3}, we prove that for $\mu$--injective systems the entropy is nonnegative and give a handy condition ensuring that it is strictly positive (Theorem \ref{thm: mu inj implies pos h}). Together with Theorem~\ref{thm: pos h}, this proves the above theorem.

In Section~\ref{sec: 4}, we provide explicit examples of noninvertible $\mu$--injective SDSs that admit a unique invariant probability measure (Corollary \ref{cor-pract-cond}).

Finally, in the last part  of this note, we present an example of a random system without $\mu$--injectivity that possesses different invariant ergodic measures.

\section{Generalized Entropy  for piecewise monotone SDS}\label{sec: 2}

Let $\MM=\MM([0,1])$ be the set of piecewise strictly monotone continuous functions from $[0,1]$ to $[0,1]$ with finitely many pieces. For a function $g \in \MM$, we denote by $\II(g):=\{I_i^g\}_{i=1}^{N(g)} =\{[d_{i-1}^g,d_{i}^g]\}_{i=1}^{N(g)}$ a finite collection of closed intervals with disjoint interiors that cover $[0,1]$ and such that $g$ is monotone on each $I_i^g$. Observe that for every $I_I^g\in \II(g)$ there exists a homeomorphisms $\gamma_i^g$ on $[0, 1]$ such that $g(x)=\gamma_i^g(x)$ for all $x\in I_i^g$.
Let $\DD(g)$ denote  the set of all points where $g$ changes slope, i.e. $\DD(g)=\bigcup_{I\in\II(g)}\partial I=\{d_0^g,d_1^g, \ldots , d_{N(g)}^g\}$.

\subsection{Pushforward and pushbackward measures}

Let $\nu$ be a Borel probability measure on $[0,1]$, and let $g\in \MM$.
\medskip
We define the \textit{pushforward measure} $g\nu$ as
$$
g\nu(A):=\nu(g^{-1}(A)) \quad \text{for } A\in\mathcal{B}([0,1]),
$$ 
where $g^{-1}(A)$ denotes the preimage of $A$ and $\mathcal{B}([0,1])$ denotes the Borel $\sigma$-algebra on $[0,1]$.

Note that if $g\in \MM$ the images $g(A):=\{g(x)\, : \, x\in A\}$, $A\in\mathcal{B}([0,1])$, are Borel measurable and we can define a \textit{pushbackward set function} $g^{-1}\nu$ on Borel sets as
$$
g^{-1}\nu(A):=\nu(g(A))   \quad \text{for } A\in\mathcal{B}([0,1]).
$$
In general, $g^{-1}\nu$ is not a measure, since it is not $\sigma$--additive (if $A_1$ and $A_2$ are two disjoint sets such that $g(A_1)=g(A_2)$ then $g^{-1}\nu(A_1\cup A_2)=g^{-1}\nu(A_1)=g^{-1}\nu(A_2)$).
However, if $g$ coincides with the homeomorphism $\gamma^g_i$ on an interval $I_i\in \II(g) $, then $g^{-1}\nu$ locally coincides with the measure   $\nu^g_i(A):= (\gamma^g_i)^{-1}\nu(A\cap I_i^g)$ (the restriction of $(\gamma^g_i)^{-1}\nu$ to $I_i^g$) in the sense that 
\begin{equation}\label{eq-loc-def-nu-g}
	g^{-1}\nu(A)=(\gamma^g_i)^{-1}\nu(A)=\nu^g_i(A) \quad \text{for}\quad A\subseteq I_i^g.
\end{equation}
Furthermore,  $g^{-1}\nu$ is bounded above by the measure $\overline{\nu}^g$  obtained as the sum of the measures $\nu^g_i$ 
\begin{equation}\label{eq-sup-nu-g}
	g^{-1}\nu(A)\leq \sum_{i=1}^{N(g)} \nu_i^g(A)=:\overline{\nu}^g(A) \quad \text{for all } A\in \mathcal{B}([0,1]).
\end{equation}


\subsection{Radon-Nikodym derivative for finite measures on $[0,1]$}

Let $\lambda$ and $\nu$ be two probability measures on the interval $[0,1]$. Then, the Radon-Nikodym derivative of $d\lambda/d\nu$ is defined as
$$\frac{\d\lambda}{\d\nu}(x):=\lim_{r\to 0}\frac{\lambda([x-r,x+r])}{\nu([x-r,x+r])} \quad \text{for $\nu$--almost every $x\in[0,1]$}.$$
The limit exists $\nu$-almost everywhere, and the Lebesgue decomposition of $\lambda$ is given by $\d\lambda=\frac{\d\lambda}{\d\nu}\d\nu+\d\lambda_{\text{sing}}$ with $\lambda_{\text{sing}}\perp\nu$. In particular,
$$
\int_A \frac{\d\lambda}{\d\nu}(x)\mathrm{d}\nu(x)\leq \lambda(A) \quad \text{for all Borel sets} \ A\subseteq [0,1],
$$
with equality if and only if $\lambda \ll \nu$ (see, for instance, \cite[Chapter 2]{Mat}).

For any $x\in (0,1)$ we denote  
\begin{equation}\label{e1_12.03.24}
\JJ^x:=\{I: \text{ $I\subset [0, 1]$ is a closed interval and}\,\,  x\in \mathrm{int}(I)\}.
\end{equation}
In the sequel, we will need the following result concerning Radon-Nikodym derivatives:
\begin{prop}\label{prop-RN-deri-NC-int}
	Let $\lambda$ be a finite Borel measure, and let $\nu$ be a Borel probability measure on the interval $[0,1]$. 
	\begin{enumerate}
		\item For $\nu$--almost every $x\in[0,1]$ we have 
		$$
		\frac{\d\lambda}{\d\nu}(x)=\lim_{\delta\to 0}\sup_{I\in \JJ^x,\, |I|<\delta}\frac{\lambda(I)}{\nu(I)} \quad .
		$$
		\item Let $Q^*_\nu(\lambda,x):=\sup_{I\in \JJ^x} \frac{\lambda(I)}{\nu(I)}$. Then 
		$$\int  \ln^+ Q^*_\nu(\lambda,x)\mathrm{d}\nu(x)\leq 2\lambda([0,1]).$$
	\end{enumerate}
\end{prop}

This result is well-known in the case $\nu$ is the Lebesgue measure or for general measures when the supremum is taken over centered intervals (see \cite[Chapter 2]{Mat} or \cite[Proposition 5]{Led86}). The stated result (for general probability measures and non-centered intervals) seems to be a folklore theorem, but since we could not find a precise reference, we provide its proof in the appendix.

\subsection{Radon--Nikodym derivative of $g^{-1}\nu$}

Since the set function $g^{-1}\nu$ is locally a measure, we can define a \textit{generalized Radon--Nikodym derivative} for $\nu$-almost every $x$ in the interior of $I_i^g\in\II(g) $ by  
$$ 
\d_{\nu}g(x)=\frac{\d g^{-1}\nu}{\d\nu}(x):=\lim_{r\to 0}\frac{\nu(g([x-r,x+r]))}{\nu([x-r,x+r])} \quad \text{for $\nu$--a.e. } x\not\in \DD (g).
$$
This limit is almost everywhere well--defined since "locally" $g^{-1}\nu$ can be viewed as a measure, and  
$$ 
\d_{\nu}g(x)=\frac{\d\nu_i^g}{\d\nu}(x) \quad \text{for $\nu$--a.e. } x\in \mathrm{int}({I_i^g})=(d_{i-1}^g,d_i^g).$$

\subsection{Entropic criterion for characterizing ergodic measures}

Let $\mu$ be a Borel probability measure supported on $\MM$, and let $\nu$ be a Borel probability measure \textit{with no atoms}. Then, for $\mu$--almost every $g$, the derivative $\d_\nu g(x)$ is well-defined for $\nu$--almost every $x\in [0,1]$ (in fact, $\nu(\DD(g))=0$). 

If $\ln^+ \d_\nu g(x)$ is $\nu\times\mu$-integrable, we can define the \textit{(generalized) Furstenberg entropy} as
$$ 
h_{\mu}(\nu):=- \int_{\MM\times [0,1]} \ln (\d_\nu g(x))\  \d\mu(g) \d\nu(x)\in (-\infty,+\infty].
$$ 
Other then in the case when $\mu$ is supported on invertible functions, for general probability measures $\mu$ on $\MM$, the entropy can be negative. However, we will see in this section that if one can guarantee that the entropy is strictly positive, then the SDS has nice contraction properties, ensuring that, in some sense, ergodic invariant measures are determined by their support.

Recall that a $\mu$--invariant measure $\nu$ is called \textit{ergodic} if for every  subset $A\subset [0, 1]$ such that $\nu_A$, the restriction of $\nu$ to $A$, is $\mu$--invariant, we have either $\nu(A)= 0$ or $\nu([0, 1]\setminus A)=0$. For a more detailed survey on ergodic measures for SDS, we refer to the Appendix of \cite{BBS}.

Let 
$$
J(x,g):=\sup\left\{\frac{\nu(g (I))}{\nu(I)}\, : \, I\in\JJ^x \right\}\in[0,+\infty]
$$
with the convention $\frac{0}{0}=0$.  Observe that by definition $\d_\nu g(x)\leq J(x,g)$. Thus if $\ln^+J$ is $\nu\times\mu$--integrable then the entropy is well--defined and we have the following:

\begin{thm}[Entropic criterion for ergodic measures]\label{thm: pos h}
	Let $\mu$ be a probability measure on $\MM$. Let $\eta$ and $\nu$ be $\mu$--invariant, ergodic Borel probability measures on $[0,1]$. Suppose that $\nu$ is atomless and $\supp\nu\subseteq \supp\eta$. If $\ln^+J\in L^1([0,1]\times \MM, \nu\times \mu)$ and $h_{\mu}(\nu)>0$, then we have $\nu=\eta$.
\end{thm}
This theorem is a direct consequence of the following result that ensures that for a given SDS the positive entropy implies that the system contracts small intervals at exponential rate .
\begin{prop}[Contractivity]\label{contractivity} 
	Let $\mu$ be a Borel probability measure on $\MM $, and let $\nu$ be a $\mu$--invariant atomless ergodic Borel probability measure on the interval $[0,1]$ such that $\ln^+ J$ is $\nu\times\mu$--integrable and  
	$h_{\mu}(\nu)>0$.
	Then for every ${h}\in (0, h_{\mu}(\nu))$ and
	$\mu^{\mathbb{N}}$-a.e. $\omega=(g_n)_{n\in\mathbb{N}}\in \MM^{\mathbb{N}}$ and $\nu$--a.e. $x\in [0,1]$ there exists a closed interval $I=I(\omega,x)$ with $x\in \mathrm{int}(I)$ such that
	$$
	|g_n\circ \ldots \circ g_1(I)|\leq \exp(-n\cdot {h}), \quad \forall n\in\mathbb{N}.
	$$
\end{prop}
D. Malicet \cite{M} proved this result when $g$ are homeomorphisms of the circle. With some precaution, his technique can be easily adapted to the case of piecewise monotone functions. For completeness, a proof is given in the appendix.

\begin{proof}[Proof of Theorem \ref{thm: pos h}]
	Assume that $\nu$ is an ergodic, $\mu$--invariant, atomless Borel probability measure on $[0,1]$ with $\supp(\nu)\subseteq \supp(\eta)$.
	Assume $\nu\neq \eta$. Then there exists a Lipschitz function $f$ on $[0,1]$ such that
	$$
	\int_{[0, 1]} f \, \mathrm{d}\nu \neq \int_{[0, 1]} f \, \mathrm{d}\eta.
	$$  
	By Birkhoff's Ergodic Theorem, for $\mu^{\otimes\mathbb{N}}$--a.e. $\omega=(g_1,g_2,\ldots)\in\MM^\N$ we have
	\begin{equation}\label{eq 2 erg}
		\lim_{n\to\infty}\frac{1}{n}\sum_{i=1}^n f(g_i\circ\cdots\circ g_1(x))=\int_{[0, 1]} f \, \mathrm{d}\nu \quad \text{for $\nu$-a.e. $x\in [0,1]$}.
	\end{equation} 
	Fix ${h}\in (0, h_{\mu}(\nu))$.  By Proposition \ref{contractivity}, there exist $x_0\in \supp\nu$ and $\Xi= \Xi(x)\in \mathcal{B}(\MM^{\mathbb N})$ with $\mu^{\otimes\mathbb N}(\Xi)>0$ such that, for all $\omega=(g_1, g_2,\ldots)\in\Xi$, convergence  (\ref{eq 2 erg}) holds and  
	\begin{equation}\label{eq contracting}
		|g_n\circ \cdots \circ g_1(I)|\leq \exp(-n\cdot {h})
	\end{equation}
	for some neighborhood $I=I(x_0,\omega)$ of $x_0$. Moreover, since $x_0\in \supp(\nu)\subseteq \supp(\eta)$, we have  $\eta(I)>0$. Thus, there exists $y \in I$ such that 
	$$
	\lim_{n\to\infty}\frac{1}{n}\sum_{i=1}^n f(g_i\circ\cdots\circ g_1({y}))=\int_{[0, 1]} f\,  \mathrm{d}\eta \neq \int_{[0, 1]} f\,  \mathrm{d}\nu,
	$$ 
	for $\mu^{\otimes\mathbb{N}}$--a.e. $(g_1,g_2,\ldots)\in\Xi$, which is impossible because $f$ is a Lipschitz function and 
	$$
	|g_n\circ \cdots \circ g_1(x)-g_n\circ \cdots \circ g_1(y)|\leq |g_n\circ \cdots \circ g_1(I)|\to 0\quad\text{as $n\to\infty$},
	$$
	by condition (\ref{eq contracting}). The proof is complete.
\end{proof}

\section{$\mu$--injectivity and positive entropy}\label{sec: 3}
In the previous section, we saw that positive entropy can be used to prove the stability of SDSs. However, calculating the entropy can be challenging, so we want to provide in this section a more manageable condition.

\subsection{Cardinality of pre-images}
For $g\in \MM $ and $x\in[0,1]$, let us denote 
$$
{n}(g,x):=\# (g^{-1}(x)),
$$ 
i.e., the number of preimages of $x$ under the map $g$, with the convention $\# ( \emptyset ):=0$. Note that $g\mapsto n(g,x)$ is measurable for every $x\in[0,1]$. 
Observe also that if $\{I_i^g\}_{i=1}^{n(g)}$ is a collection of covering intervals such that $g$ is monotone on each of the pieces, then we have
\begin{equation}\label{sum}
	n(g, x)=\sum_{i=1}^{n(g)}{\bf 1}_{g (I_i^g)}(x)\qquad\text{for all } x\not\in \DD(g).
\end{equation} 
The following lemma shows that if the number of preimages is bounded in mean, then the entropy of the system is well--defined. 
\begin{lem}
	If $\nu$ has no atoms and $n\in L^1(\MM\times[0,1],\mu\times\nu)$, then $\ln^+J\in L^1(\MM\times[0,1],\mu\times\nu)$ and $h_\mu(\nu)$ is well-defined. 
\end{lem}
\begin{proof}
	Observe that by (\ref{eq-sup-nu-g}) 
	\begin{align*}
		J(g,x)&=\sup_{I\in\JJ^x}\frac{\nu(g(I))}{\nu(I)}\leq \sup_{I\in\JJ^x} \frac{\overline{\nu}^g(I)}{\nu(I)}=: Q^*_\nu(\overline{\nu}^g,x),
	\end{align*}
	where $\JJ^x$ is given by (\ref{e1_12.03.24}). Thus by Proposition \ref{prop-RN-deri-NC-int} and (\ref{sum}),
	\begin{align*} 
		\int_{[0,1]} \ln^+ J(g,x) \d\nu(x) &\leq  \int_{[0,1]} \ln^+  Q^*_\nu(\overline{\nu}^g,x) \d\nu(x)\\
		&\leq 2\overline{\nu}^g([0,1])=2 \sum_{i\in\II(g)} \nu_i^g([0,1]) \\
		&= \sum_{i\in\II(g)}\nu(g(I_i^g))= 2\int_{[0,1]} \sum_{i\in \II(g)}{\bf 1}_{g (I_i^g)}(x) \d\nu(x) \\
		&= 2\int_{[0,1]} n(g,x) \d\nu(x),
	\end{align*}
	since $\nu(\DD(g))=0$.	 
\end{proof}

\subsection{Entropy of $\mu$--injective SDS}

\begin{definition}\label{def: mu-inj} We say that a stochastic dynamical system $(\MM,\mu)$ is \textit{$\mu$--injective in $x\in [0,1]$} if 
	\begin{equation}\label{eq-mu-inj}
	\int_{\MM} n(g,x) \, \d \mu(g) \leq  1 .
\end{equation}
\end{definition}

\begin{prop}\label{prop: h pos} Let $(\MM,\mu)$ be a stochastic dynamical system, and let $\nu$ be an atomless Borel probability measure on $[0,1]$. If  $(\MM,\mu)$ is $\mu$--injective for $\nu$-almost all $x\in [0,1]$,
	then either $h_{\mu}(\nu)>0$ (possibly infinite) or 
	$\d_{\nu}g(x)\equiv 1$ for $\nu$--a.e. $x\in [0, 1]$ and  $\mu$--a.e. $g\in\MM$.
\end{prop}

\begin{proof} Since $g^{-1}\nu$ coincides with the measure $\nu^g_i:= (\gamma_i^g)^{-1}\nu\vert_{I_i^g}$ on each $I_i^g\in I(g)$, we have
	\begin{equation}\label{int n}
		\begin{aligned}
			\int_{[0, 1]} \d_{\nu} g(x) \, \d\nu(x)&=\sum_{i\in \II(g)} \int_{I_i^g} \frac{\d g^{-1} \nu}{\d \nu}(x) \, \d\nu(x) \\
			&= \sum_{i\in \II(g)} \int_{I_i^g} \frac{\d\nu_i^g}{\d \nu}(x) \, \d\nu(x)\\
			&\leq \sum_{i\in I} \nu(g(I_i^g))\\
			&= \int_{[0, 1]} n(g,x)\, \d\nu(x),
		\end{aligned}
	\end{equation}
	where the last equality follows from (\ref{sum}) and the fact that $\nu(\DD(g))=0$.
	Thus, by Jensen's inequality and Fubini's theorem,
	\begin{eqnarray*}
	h_{\mu}(\nu)&= &\int_{\MM} \int_{[0, 1]} -\ln(\d_{\nu} g(x)) \, \d\nu(x)\, \d\mu(g)\\ &\geq&	-\ln\Big( \int_{\MM} \int_{[0, 1]} \d_{\nu} g(x) \, \d\nu(x)\, \d\mu(g)\Big)\\&
	\geq &
	-\ln\Big(
	\int_{[0, 1]} \int_{\MM} n(g,x) \, \d\mu(g)\, \d\nu(x)\Big).
		\end{eqnarray*}	
Now let us assume that (\ref{eq-mu-inj}) holds for $\nu$--a.e $x$. Then, as above by the Jensen inequality, we have $h_{\mu}(\nu)\geq -\ln(1)= 0$ and this inequality is strict unless $\d_{\nu}g(x)\equiv c$ for $\nu$--a.e. $x\in X$, for $\mu$--almost every $g\in\MM$. In the latter case, we obtain again by (\ref{int n}) and (\ref{eq-mu-inj}) that 
$$
c= \int_{\MM} \int_{[0, 1]} d_{\nu}g(x)\, \d\nu(x)\, \d\mu(g)\leq \int_{[0, 1]} \int_{\MM}n(g,x)\, \d\mu(g)\, \d\nu(x)\leq 1.
$$ 
Now, if $c<1$ then $h_{\mu}(\nu)>0$. In the remaining case  $c=1$. Hence $\d_{\nu}g(x)=1$ for $\nu$--a.e. $x\in [0, 1]$, for $\mu$--almost every $g\in\MM$.
\end{proof}

Proposition \ref{prop: h pos} should be considered a generalization to $\mu$-injective SDSs of the well-known result valid for a SDS driven by invertible maps $g$. Indeed, then either $h_\mu(\nu)>0$ or $g\nu=\nu$ for $\mu$--almost every $g$ (see, for instance, \cite[Proposition 3.7]{M}).
However, in the case of a non-invertible map, the fact that $d_\nu g \equiv 1$ does not imply that $\nu$ is $g$--invariant. We have, for instance, the following example, which also provides an example of a $\mu$--injective SDS that has several (atomic) ergodic $\mu$--invariant probability measures.
\begin{ex}\label{ex: mu inj}
	Suppose that $\mu$ is supported on $\Gamma=\{\phi_1,\phi_2\}$ consisting of a $\frac{1}{2}$-tent-map and an upside-down $\frac{1}{2}$-tent-map, as shown below.

	\begin{center}
		\begin{tikzpicture}
			\draw (0,0) rectangle (2,2);
			\draw[red] (0,0) -- (1,1) -- (2,0);
			\draw (0.5,0.5) node[anchor=north]{\begin{color}{red}$\phi_1$\end{color}};
			\draw (0,0) node[anchor=north]{0};
			\draw (2,0) node[anchor=north]{1};
			\draw (0,2) node[anchor=east]{1};
			\draw (1,-0.05)--(1,0.05) node[anchor=north]{$\frac{1}{2}$};
			\draw (-0.05,1)--(0.05,1) node[anchor=east]{$\frac{1}{2}$};
			\draw[blue] (0,2) -- (1,1) -- (2,2);
			\draw (1.5,1.5) node[anchor=north]{\begin{color}{blue}$\phi_2$\end{color}};
		\end{tikzpicture}
	\end{center}
	
Assume that $\mu(\phi_i)=\frac{1}{2}$ for $i=1,2$. This system is $\mu$--injective, since
	$$
	\begin{aligned}
		\int_{\MM} n(g,x) \, \d \mu(g)
		&= 2\mu(\phi_1)1_{[0,\frac{1}{2})}(x)+(\mu(\phi_1)+\mu(\phi_2))1_{\{\frac{1}{2}\}}(x)\\	&\quad +2\mu(\phi_2)1_{(\frac{1}{2},1]}(x)=1.
	\end{aligned}
	$$
	It is easily checked that any probability measure $\nu$ symmetric with respect to $1/2$ (that is invariant under the map $\sigma: x\mapsto 1-x$) is $\mu$-invariant. For instance, the Lebesgue measure on $[0, 1]$ is $\mu$--invariant, but also the restriction of the Lebesgue measure to the invariant set $Y=[0, 1/4]\cup [3/4, 1]$.
	
	For all these measures we have $\d_\nu g\equiv 1$. In fact, since $\phi_1(I)=I$ for $I\subseteq [0,1/2]$ and $\phi_1(I)=\sigma(I)$ for $I\subseteq[1/2,1]$ (and reversely for $\phi_2$), we have
	$$
	\d_\nu\phi_i(x)=\lim_{r\to 0}\frac{\nu(\phi_i([x-r,x+r]))}{\nu([x-r,x+r])}=\lim_{r\to 0}\frac{\nu([x-r,x+r])}{\nu([x-r,x+r])}=1.
	$$
	
	However, a symmetric measure $\nu$ is not $\phi_1$-invariant, since for $I\subseteq [0,1/2]$ we have
	$$
	\phi_1\nu(I)=\nu(\phi_1^{-1}(I))=\nu(I\cup\sigma(I))=2\nu(I),
	$$
	whereas $\phi_1\nu(I)=0$ for $I\subseteq[1/2,1]$.
	
	As a matter of  fact, this SDS has infinitely many ergodic measures, which are atomic. In fact, every measure $\nu_{x_0}=\frac{1}{2}\delta_{x_0}+\frac{1}{2}\delta_{\sigma(x_0)}$ for $x_0\in[0,1/2]$ is ergodic.	
\end{ex}

The natural question arises:
\vskip2mm
\textbf{Open question:} Does there exist an SDS with an \textit{atomless} ergodic $\mu$--invariant measure $\nu$ such that $\d_\nu g \equiv 1$, but $\nu$ is not $g$--invariant?

\subsection{Practical condition to ensure $h_\mu(\nu)>0$}

	We provide in this section a hands-on condition that ensure that $d_\nu g\not \equiv 1$.

\begin{definition}\label{def: unbounded} 
	We say that  a stochastic system $(\MM,\mu)$  \textit{contracts a neighbouhood of $x_0\in[0,1]$ } if there exists $\varepsilon>0$ such that the set of $g\in\MM$ satisfying:
$$
  g(x_0)=x_0
  $$
 and 
 $$
   |g(x)-g(x_0)|<|x-x_0|\quad\text{for all}\,\, x\in [x_0-\epsilon, x_0+\epsilon]\cap[0,1]
$$
has positive $\mu$-measure.
\end{definition}
If $x_0=0$ this means that the  stochastic system  is \textit{below diagonal on a neighbourhood of $0$ } that is there exists $\varepsilon>0$ such that
$$
\mu(\{g\in\MM\, :\, g(0)=0\ \mbox{ and }\ g(x)<x\ \forall x\in (0,\epsilon]\})>0,
$$
and, analogously, if $x_0=1$, \textit{above diagonal on a neighbourhood of $1$}, that is there exists $\varepsilon>$ such that
$$\mu(\{g\in\MM\, :\, g(1)=1\ \mbox{ and }\   g(x)>x\ \forall x\in [1-\epsilon,1)\})>0.
$$ 
If $x_0\in(0,1)$, it means that it crosses the diagonal with a "slope less the one".

\begin{thm}\label{thm: mu inj implies pos h}
	Let $(\MM,\mu)$ be a stochastic dynamical system $\mu$-injective in all but countably many $x\in[0,1]$, and assume  that it contracts a neighbouhood of $x_0\in[0,1]$.
Moreover, assume that $\nu$ is an atomless $\mu$--invariant probability measure with $x_0\in\supp(\nu)$.  Then $h_{\mu}(\nu)>0$.
	
In particular if $\nu$ is ergodic,  any other $\mu$-invariant, ergodic Borel probability measure  {$\eta$ whose support contains $\supp(\nu)$ coincide with $\nu$}. 
\end{thm}

The above theorem follow immediately from Proposition \ref{prop: h pos} and the following lemma:

\begin{lem}\label{lem: gnu}
	Let $(\MM,\mu)$ be a stochastic dynamical system $\mu$-injective in all but countably many $x\in[0,1]$, and assume  that it contracts a neighbourhood of $x_0\in[0,1]$.
Moreover, assume that $\nu$ is an atomless $\mu$--invariant probability measure with $x_0\in\supp(\nu)$. Then  $$\nu\otimes\mu(\{(x,g)\in [0,1]\times \MM\, : \, d_{\nu}g(x)\neq 1\})>0.
	$$  
\end{lem} 

\begin{proof}

	Let us assume that the system contracts a neighbourhood of $x_0$.\\ Let $J_m:=[x_0-\frac{1}{m}, x_0+\frac{1}{m}]\cap [0,1]$ and  
$$
J_m^+:=[x_0, x_0+\frac{1}{m}]\cap [0,1]\text{ and }J_m^-:=[x_0-\frac{1}{m}, x_0]\cap [0,1].
$$
Since all the maps in $\MM$ are piecewise monotone, we can write 
$$
\{g\, :\, g(x_0)=x_0\}=\bigcup_{m\in\mathbb{N}} \{g\, : g(x_0)=x_0, \,\,\, g\vert_{J_m^+}\text{ and }g\vert_{J_m^-}\text{  are monotone}\}.
$$ 
	Thus by continuity of measures, there is $m\in\mathbb{N}$ such that $\mu(\Delta_m)>0$ where  $\Delta_m$ is the set of $g\in\MM$ such that
	$$  
	g\vert_{J_m^+}\text{ and }g\vert_{J_m^-}\text{  are monotone}, g(x_0)=x_0
	$$
	and  
	$$
	|g(x)-g(x_0)|<|x-x_0|\quad\text{for all}\quad x\in J_m.
	$$
 We claim that $\nu(g^n J_m)\to 0$ for any $g\in \Delta_m$. In fact  we easily check that $ g^n J_m=[a_n,b_n]$ with $a_n=\min\{g^n(x_0),g^n(x_0+ \frac{1}{m}), g^n(x_0+ \frac{1}{m})\}$ and $b_n=\max\{g^n(x_0),g^n(x_0+ \frac{1}{m}), g^n(x_0+ \frac{1}{m})\}$, and $a_n$ and $b_n$ converge to $x_0$. Thus $\nu(g^n J_m)\to\nu(\{x_0\})=0$, by atomlessness of $\nu$.

 Suppose now that $\d_\nu g(x)= 1$ for $\nu$--a.e. $x\in[0,1]$ and $\mu$--a.e. $g\in\MM$. Then for all Borel set  $A\subseteq I_i^g$ we have
\begin{equation}\label{eq-nu-d=1}
	\nu(A)=\int_A \d_\nu g(x) \d\nu(x)= \int_A \frac{\d\nu_i^g}{\d\nu}(x)  \d\nu(x)\leq \nu_i^g(A)=\nu(g(A)).
\end{equation} 
 Since for every $g\in \Delta_m$, $g(J^+_m$) is contained in either $J^+_m$ or $J^-_m$, we can iterate (\ref{eq-nu-d=1}) obtaining
 $$
 \nu(J_m^+)\leq \nu(g(J_m^+))\leq \nu(g^n(J_m^+))\leq  \nu(g^n(J_m))\to 0.
 $$
 Thus $\nu(J_m^+)=0$ and similarly $\nu(J_m^-)=0$. Thus $\nu(J_m)=0$, which contradicts the hypothesis  that $x_0\in \supp \nu$.
\end{proof}

\begin{proof}[Proof of  Theorem~\ref{thm: mu inj implies pos h}]
	By Proposition~\ref{prop: h pos}, $\mu$--injectivity implies that either $h_{\mu}(\nu)>0$ or $\mu\otimes\nu(\{(g, x)\, : \, \d_{\nu}g(x)=1\})=1$. The latter case cannot occur due to  Lemma~\ref{lem: gnu}. Thus the entropy  $h_{\mu}(\nu)$ has to be strictly positive, which was to prove.

		If  $\nu$ is also ergodic the claim follows now by Theorem~\ref{thm: pos h}.
\end{proof}

\section{Examples}\label{sec: 4}
\subsection{A criterion for unique ergodicity }
In this section, we present a criterion for the existence and uniqueness of a $\mu$--invariant probability measure.  This criterion, derived from our previous results, is not optimal but can be used to construct noninvertible uniquely ergodic SDSs.

\begin{cor}\label{cor-pract-cond}
	Let $\mu$ be a probability measure  with finite  support $\Gamma\subseteq\MM $ and let  $\Gamma^*$ the countable semigroup of $\MM$  generated by $\Gamma$. Suppose that the generated  SDS satisfies the following properties:
	\begin{itemize}
		\item The set $C:=\bigcap_{x\in (0,1)} \overline{\Gamma^*x}\subseteq[0,1]$
		is not empty, i.e.,  there is a point in the closure of all the orbits $\Gamma^*x$,  $x\in (0,1)$.
		\item  There exists $x_0\in C\cap (0,1)$ such the family of measures 
		$$\left\{\sum_{k=0}^n\pi^k(x_0,\cdot)/n\right\}_{n=1}^{\infty}
		$$
		 \textbf{is tight} in $(0,1)$.
		\item The SDS\textbf{ contracts a neighbourhood of $y_0\in C$.}
		\item The system is \textbf{$\mu$--injective} in all  $x\in(0,1)$ 
		\item Every $x\in(0,1)$ has \textbf{zero  or an infinity of preimages}, i.e.  $$
		\#\{y: g(y)=x \mbox{ for some } g\in \Gamma^*\}=\infty \text{ or } 0.$$
		
	\end{itemize}
Then there exits a unique  $\mu$--invariant probability measure   on $(0,1)$. This probability measure is atomless.   	
\end{cor}
\begin{proof}

Since $\phi$ are continuous 
the transition probability :
$$
\pi(x, A)=\int_{\MM} {\bf 1}_A(g(x))\, \d\mu(g)\qquad \mbox{ for $x\in\R$ and $A\in {\mathcal B}((0, 1))$}
$$ 		 
is a Feller operator on  $[0,1]$. Thus for any $x\in[0,1]$ all limit measure of the family $\{\sum_{k=0}^n\pi^k(x,\cdot)/n\}_n$ is  a $\mu$-invariant probability measure.
Further more since  $\{\sum_{k=0}^n\pi^k(x_0,\cdot)/n\}_n$ is tight in $(0,1)$ there exists an ergodic measure   $\nu_0$ do not charge $0$ and $1$, i.e.  $\nu_0((0,1))=1$, and such that $\supp\nu_0\subseteq \overline{\Gamma^* x_0}$.

We claim that under the hypothesis  any $\mu$-invariant portability $\nu$  such that $\nu((0,1))=1$ is atomless.  Suppose  $\nu$ has atoms  and let $M:=\sup_{y\in[0,1]}\nu(y)$ and $A:=\{x\in (0,1)| \nu(x)=M\}$. Since $\nu$ is a finite measure, $A$ is a finite set. Furthermore $\mu$-injectivity implies that $g^{-1}A\subsetneq A$ for all $g\in\Gamma^*$ and that any $a\in A$ has at least a pilgrimage. In fact let $a\in A$ then
$$M=\nu(a)=\sum_{g\in\Gamma}\mu(g)\sum_{y\in g^{-1}(a)}\nu(y)\leq \sum_{g\in\Gamma}\mu(g)M n(g,a)\leq M,$$
where the equality may hold only if $\nu(y)=M$ for all $y\in  g^{-1}(a)$ and $\sum_{g\in\Gamma}\mu(g) n(g,a)=1$.
This contradict the fact that  all $x\in(0,1)$ have zero or infinitely many preimage under $\Gamma^*$. 

Let $\nu$ be an ergodic invariant probability measure. 	
Observe  $\supp\nu$ is a closed $\Gamma^*$--invariant, i.e., $g(\supp\nu)\subseteq \supp\nu$ for all $g\in\Gamma^*$ \footnote{In fact, since
	$$1=\nu(\supp\nu)=\sum_{g\in\Gamma}\mu(g)\nu(g^{-1}\supp\nu),
	$$
	then $\nu(g^{-1}\supp\nu)=1$ for every $g\in\Gamma$, hence  $g^{-1}\supp\nu\supseteq\supp\nu$, by the fact that $g^{-1}\supp\nu$ is closed and $\supp\nu$ is minimal. Thus $\supp\nu \supseteq g(g^{-1}\supp\nu)\supseteq g(\supp\nu)$.}, thus  $C\subseteq \supp\nu$.  In particular  $x_0\in \supp\nu$ and $\overline{\Gamma^*x_0}\subseteq \supp\nu$.
Thus $\supp \nu_0\subseteq \supp \nu$ and since the condition  of Theorem \ref{thm: mu inj implies pos h} are satisfied then $\nu=\nu_0$.
\end{proof}

\begin{example} Consider the two  examples below, where the measure $\mu$ is equidistributed on the functions in each graphic 
\begin{center}

\includegraphics*[width=0.45\textwidth]{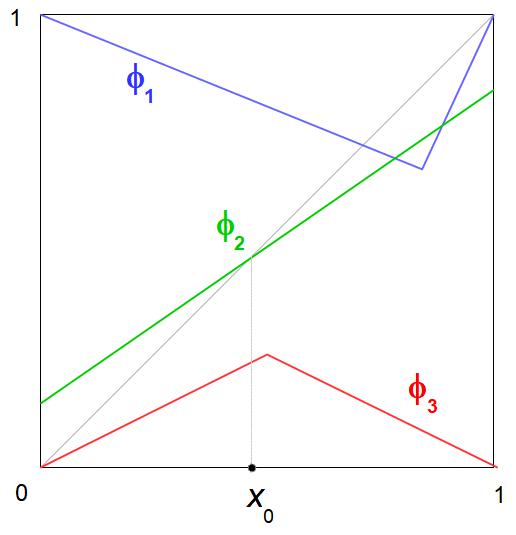}\quad \includegraphics*[width=0.45\textwidth]{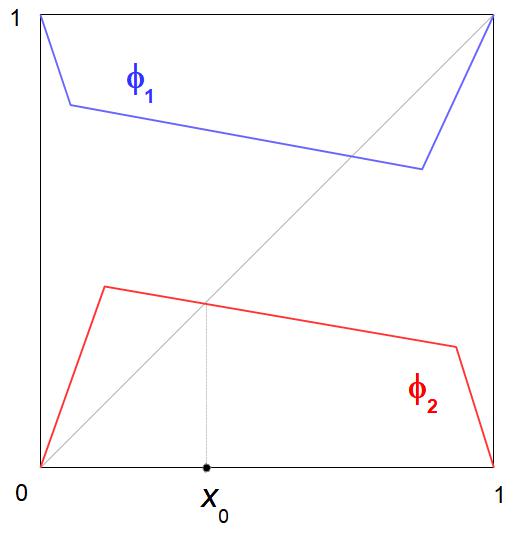}

\end{center}
 It can be checked that this two SDSs  satisfy the hypothesis of Corollary \ref{cor-pract-cond}, thus admit a unique invariant probability on $(0,1)$.

In both cases $x_0=y_0$ is in $C$ since $\lim_n\phi_2^n(x)=x_0$ for all $x\in(0,1)$. The family $\{\sum_{k=0}^n\pi^k(x_0,\cdot)/n\}_{n=1}^{\infty}$ is tight since the SDS is strongly repelled at $0$ and $1$ (see for instance 
\cite[Proposition 9.1]{ABB}) .
\end{example}

\subsection{Construction of SDS without unique ergodicity}\label{sec: ex}
Let us give here an example where unique ergodicity fails in the setting where 
 $\mu$--injectivity does not hold, but the systems is below diagonal on a neighbourhood of $0$ and above diagonal on a neighbourhood of $1$.

We start with a general observation and construction.

First observe that every semigroup of continuous functions acts on the set of closed subintervals of $[0,1]$ by the map $I\mapsto g(I)$

 Let $\JJ$ denote some family of closed subintervals of $[0, 1]$. Let $\mu$ be a Borel probability measure on $\MM$ such that $g\JJ\subseteq \JJ$  for $\mu$--almost every  $g\in\MM$.  We are given a family $\{\nu_I: I\in \JJ\}$ of Borel probability measures on $[0, 1]$. We assume that this family is equinvariant in the sense that   
\begin{equation}\label{eq-equiinv-nuI}
\nu_{g(I)}=g\nu_I\qquad\text{$\bar\nu\otimes\mu$--a.s.}
\end{equation}
and such that $I\mapsto \nu_{I}(A)$ is measurable for all $A\in\mathcal{B}([0,1])$.

We start with a simple lemma. 

\begin{lem}\label{lemma nu} Let $(\JJ, \Xi, \bar\nu)$  be a  measure space. Assume that $\bar\nu$ is an $\mu$--invariant probability, i.e., it satisfies
\begin{equation}\label{e1_16.06.23}
\int_{\JJ\times \Gamma} f(g(I))  \d\bar\nu(I))\, \d\mu(g)=\bar\nu(f(I)) \text{ for any  } f\in L^\infty(\JJ, \Xi, \bar\nu).
\end{equation}
Then the measure
$$
\nu(A):=\int_{\JJ}\nu_I(A)\, \d\bar\nu(I)\quad\text{for any Borel } A\subset[0,1]
$$
is a $\mu$--invariant measure on $[0,1]$.

\end{lem}

\begin{proof} For  $A\in\mathcal B([0, 1])$ we have
$$
\begin{aligned}
&\int_{\MM}\nu(g^{-1}(A))\, \d\mu(\d g)=
\int_{\MM}\left(\int_\JJ \nu_I(g^{-1}(A)) \, \d\bar\nu( I)\right)\, \d\mu(g)\\
&=\int_{\MM}\left(\int_\JJ \nu_{g(I)}(A)\, \d \bar\nu(I)\right)\, \d\mu( g)
\stackrel{\eqref{e1_16.06.23}}{=} \int_\JJ \nu_I(A)\, \d\bar\nu( I)=\nu(A),
\end{aligned}
$$  
which ends the proof.
\end{proof}

We are now in a position to give a concrete example of an SDS where unique ergodicity fails, i.e., we will construct two different  $\mu$--invariant probability measures $\nu_1$, $\nu_2$.  On $[0,1]$ consider the following three continuous maps 
$$
\phi_1(x):=
\begin{cases}
3x&\text{ if } x\in[0,\frac{1}{3}],\\
-3x+2 &\text{ if } x\in(\frac{1}{3}, \frac{2}{3}],\\
3x-2 &\text{ if } x\in(\frac{2}{3}, 1],
\end{cases} \quad \ \ \ \phi_2(x):=\frac{x}{3}   \quad\text{ and } \quad \phi_3(x):=\frac{x}{3}+\frac{2}{3}.
$$

\begin{center}
\begin{tikzpicture}
\draw (0,0) rectangle (3,3);
\draw (0,0) -- (1,3) -- (2,0) -- (3,3);
\draw (0,0) node[anchor=north]{0};
\draw (3,0) node[anchor=north]{1};
\draw (0,3) node[anchor=east]{1};
\draw (1,-0.05)--(1,0.05) node[anchor=north]{$\frac{1}{3}$};
\draw (2,-0.05)--(2,0.05) node[anchor=north]{$\frac{2}{3}$};
\draw (1.5,2.5) node[anchor=north]{$\phi_1$};
\end{tikzpicture}
\qquad
 \begin{tikzpicture}
\draw (0,0) rectangle (3,3);
\draw (0,0) -- (3,1);
\draw (0,0) node[anchor=north]{0};
\draw (3,0) node[anchor=north]{1};
\draw (0,3) node[anchor=east]{1};
\draw (-0.05,1)--(0.05,1) node[anchor=east]{$\frac{1}{3}$};
\draw (1.5,1.1) node[anchor=north]{$\phi_2$};
\end{tikzpicture}\qquad
\begin{tikzpicture}
\draw (0,0) rectangle (3,3);
\draw (0,2) -- (3,3) ;
\draw (0,0) node[anchor=north]{0};
\draw (3,0) node[anchor=north]{1};
\draw (0,3) node[anchor=east]{1};
\draw (-0.05,2)--(0.05,2) node[anchor=east]{$\frac{2}{3}$};
\draw (1.5,2.5) node[anchor=north]{$\phi_3$};
\end{tikzpicture}
\end{center}

\begin{prop}\label{prop ex}
Let $\mu$ be a probability measure on $\MM$ such that    $\mu(\phi_1)=p$, $\mu(\phi_2)=\frac{1-p}{2}=\mu(\phi_3)$ for some $p\in (\frac{1}{2},1)$. Then there exist at least two different $\mu$--invariant ergodic 
probability measures $\nu_1,\nu_2$ on $(0,1)$. 
\end{prop}

\begin{remark} Observe that the system is below diagonal on a neighbourhood of $0$ and above diagonal on a 
neighbourhood of $1$, but $\mu$--injectivity does not hold, because $\int_{\MM}n(g,x)\, d\mu(g)\geq 3 p >1$ for all $x\neq 0,1$.
\end{remark}

\begin{proof}[Proof of Proposition~\ref{prop ex}]
Let $$K_n:=\{k\in\N: k=k_0+k_13+\cdots +k_{n-1}3^{n-1} \text{ with } k_i=0\text{ or }  2\}.$$	
Consider the collection of all subintervals of $[0,1]$ which are used in the construction of the Cantor--$\frac{1}{3}$-set, i.e.,
$$\mathcal{C}:=
\Big\{\Big[\frac{k}{3^{n}}, \frac{k+1}{3^{n}}\Big]\ : \ k\in K_n\ , \ \forall n\in\mathbb{N}\Big\}.
$$
Observe that $\phi_i(\mathcal{C})\subseteq \mathcal{C}$ for $i=1,2,3$, hence $\MM$ acts as a semi-group on the discrete set $\mathcal{C}$. Moreover, we can define a $\mu$-invariant probability measure $\bar\nu$ on $(\mathcal{C}, \mathcal{P}(\mathcal{C}))$  - where $\mathcal{P}(\mathcal{C})$ denotes the power set of $\mathcal{C}$ - as follows: 
$$\bar\nu(I):=c\Big(\frac{a}{2}\Big)^{\lev(I)}
\qquad \text{ for } I\in\mathcal{C},
$$
where $\lev(I)$ denotes the level of the interval in the Cantor-construction (i.e. $\lev([\frac{k}{3^n}, \frac{k+1}{3^n}]):=n$) and 
$a:=\frac{1-\sqrt{1-4p(1-p)}}{2p}$ and $c:=\big(\sum_{n=0}^{\infty} {a}^n \big)^{-1} 
$. Note that the sum in the latter expression is finite due to the choice of $p$.
Let us verify that $\bar\nu$ is $\mu$-invariant:
$$
\begin{aligned}
p\bar\nu(\phi_1^{-1}(I))+\frac{1-p}{2}\bar\nu(\phi_2^{-1}(I))+\frac{1-p}{2}\bar\nu(\phi_3^{-1}(I)) &
\overset{?}{=}
\bar\nu(I)
\\
\iff \qquad
2p\Big(\frac{a}{2}\Big)^{\lev(I)+1 } +\frac{1-p}{2} \Big(\frac{a}{2}\Big)^{\lev(I)-1} 
&\overset{?}{=}\Big(\frac{a}{2}\Big)^{\lev(I)},
\end{aligned}
$$
which is satisfied by the choice of a.

Now, using the construction from Lemma~\ref{lemma nu} for the above measure $\bar\nu$, we will construct two different $\mu$--invariant ergodic probability measures on $[0,1]$. First, let us consider the uniform probability measures $\nu_I$ on $I$ given by $$\d\nu_I:=\frac{1}{\vert I\vert } {\bf 1}_{I}(x)\, \d x, \qquad \text{ for } I\in \mathcal{C}.
$$
Observe that  
$$ \nu_{\phi_i(I)}= \phi_i \nu_I, \qquad \forall i=1,2,3.$$ 
Indeed, for $f$ being a continuous map on $[0,1]$, we have
$$
\begin{aligned}
\nu_{\phi_1(I)}(f)&=\int_{[0, 1]} f \, \d\nu_{\phi_1(I)}=\frac{1}{3\vert I \vert}\int_{\phi_1(I)} f(x) \, \d x\\
&=
\frac{1}{3\vert I \vert}\int_{I} f(\phi_1(y)) \, 3 \d y=\nu_{I}(\phi_1^{-1}(f))
\end{aligned}
$$
and similarly for $\phi_2, \phi_3$.
Therefore, we can apply Lemma~\ref{lemma nu} and obtain that the measure $$\nu_1:=\sum_{I\in \mathcal{C}} \bar\nu(I) \nu_I
$$ is $\mu$--invariant. Moreover, $\nu_1$ is absolutely continuous with respect to the Lebesgue measure, $\nu_1 \ll \leb$, and $\supp(\nu)=[0,1]$.

In order to construct a second (different) $\mu$--invariant probability measure $\nu_2$ on $(0,1)$, let us first consider the Cantor measure $\eta$ on $[0,1]$, i.e. the unique probability measure such that $\eta=\frac{1}{2}\eta\circ \phi_2^{-1} +\frac{1}{2}\eta\circ \phi_3^{-1}$, see for instance \cite{LM}. 
 In particular, $\eta([\frac{k}{3^n}, \frac{k+1}{3^n}])=2^{-n}$ for $k\in K_n$.

 We set $$\widetilde{\nu}_I:=\frac{\eta( \cdot \cap I)}{\eta(I)} \qquad \text{ for } I\in \mathcal{C}.$$
Again, \begin{equation}\label{equiv eta} \widetilde{\nu}_{\phi_i(I)}=\phi_i\widetilde{\nu}_I \qquad \text{ for }i=1,2,3.\end{equation} Indeed, for $I\neq [0,1]$,
 $I\in\mathcal{C}$, we obtain for all $A\in\mathcal{B}[0,1]$ 
$$
\begin{aligned}
 \widetilde{\nu}_{\phi_1(I)}(A)= &\frac{\eta( A \cap \phi_1(I))}{\eta(\phi_1(I))}=\frac{2^{-1}\eta(A\cap\phi_1( I))}{\eta(I)}=
\frac{\eta( \phi_1^{-1}(A \cap \phi_1(I))\cap[0,\frac{1}{3}])}{\eta(I)}
\\=& 
\frac{\eta( \phi_1^{-1}(A) \cap I)}{\eta(I)}=\widetilde{\nu}_I(\phi_1^{-1}(A))=\phi_1\widetilde{\nu}_I(A),
\end{aligned}
$$
using that either $I\subseteq [0,\frac{1}{3}]$ or $I\subseteq [\frac{2}{3},1]$ that the size of $\phi^{-1}(A)$ is equally distributed among each of the intervals.
Now, if $I=[0,1]$ then $ \widetilde{\nu}_{I}(\phi_1^{-1}(A))=\eta(\phi_1^{-1}(A))= 2\frac{1}{2}\eta(A)=\eta(A)= \widetilde{\nu}_{\phi_1(I)}(A)$.
Further,  since $\phi_j$  is injective for $j=2,3$ we see that for all $I\in \mathcal{C}$,
$$
\begin{aligned} 
\widetilde{\nu}_{\phi_j(I)}(A)&= \frac{\eta( A \cap \phi_j(I))}{\eta(\phi_j(I))}= \frac{2\eta( A \cap \phi_j(I))}{\eta(I)}\\
&= \frac{\eta( \phi_j^{-1}(A \cap \phi_j(I)))}{\eta(I)}=\widetilde{\nu}_I(\phi_j^{-1}(A)),
\end{aligned}
$$ 
which proves (\ref{equiv eta}).

Thus, by Lemma~\ref{lemma nu},  
$$\nu_2:=\sum_{I\in\mathcal{C}}\bar\nu(I)\widetilde{\nu}_I
$$ is a $\mu$--invariant probability measure, which by construction is absolutely continuous with respect to the Cantor measure $\eta$ and $\supp(\nu_2)=\partial \mathcal{C}$ is the Cantor-$\frac{1}{3}$-set. 
Moreover, $\nu_2$ is atomless, thus we can view it as a probability measure on $(0,1)$.
 The $\mu$--invariant probability measures $\nu_1$, $\nu_2$ itself might not be ergodic, but we obtain the existence of two different $\mu$--invariant ergodic Borel probability measures by 
 Rohklin's ergodic decomposition theorem.
 \end{proof}

\section{Appendix}
\subsection{Proof of Proposition \ref{prop-RN-deri-NC-int}}
\begin{myprop}{\ref{prop-RN-deri-NC-int}}
	Let $\lambda$ be a finite Borel measure, and let $\nu$ be a Borel probability measure on the interval $[0,1]$. 
	\begin{enumerate}
		\item For $\nu$--almost every $x\in[0,1]$ we have 
		$$
		\frac{\d\lambda}{\d\nu}(x)=\lim_{\delta\to 0}\sup_{I\in \JJ^x,\, |I|<\delta}\frac{\lambda(I)}{\nu(I)} \quad .
		$$
		\item Let $Q^*_\nu(\lambda,x):=\sup_{I\in \JJ^x} \frac{\lambda(I)}{\nu(I)}$. Then 
		$$\int  \ln^+ Q^*_\nu(\lambda,x)\mathrm{d}\nu(x)\leq 2\lambda([0,1]).$$
	\end{enumerate}	
\end{myprop}
The first part of the  proposition is a consequence of the following version of the non-centered Vitali covering Theorem for a general finite measure on the interval:
\begin{lem}\label{lem-vitali}
	Let $\nu$ be a finite regular measure on $[0,1]$, and let $A\subseteq [0,1]$ be a Borel set. Let $\JJ$ be a family of closed interval such that 
	$$
	\inf\left\{|I|: I\in\JJ \mbox{ and } x\in \mathrm{int}(I)\right\}=0\quad \forall x\in A.
	$$
	Then there exists a countable family $\{I_i\}\subset\JJ$ of disjoint intervals such that $\nu(A\setminus\bigcup_iI_i)=0$.
	 
\end{lem}	
\begin{proof}
\textbf{Step 1.} Let $K$ be a compact set, and let $U\supset K$ be open in $[0,1]$. We claim that there exists a finite family of disjoint intervals $\{I_i\}\subset \JJ$ such that $I_i\subset U$ and $\nu(K\setminus\bigcup_iI_i)\leq \frac12\nu(K)$.

In fact, since the intervals of $\JJ$ included in $U$ form a covering of the compact set $K$,  we can find a finite covering family $\{\hat I_j\}_{j=1}^N$. Passing to a subfamily if necessary, we can assume that every interval in this family intersects at most two other interval. Additionally, we may assume that
$$
\hat I_i\cap \hat I_j=\emptyset \mbox{  if } j\neq i\pm 1.
$$
Consider the family   $I'_i=\hat I_{2i}$ for $i=1\ldots \lfloor N/2\rfloor$  and $I''_i=\hat I_{2i+1}$ for $i=1\ldots \lfloor N/2\rfloor$. Then $\II'$ and $\II''$ are disjoint families and since their union covers $K$, either $\nu(K\setminus\bigcup_iI'_i)< \frac12 \nu(K)$ or $\nu(K\setminus\bigcup_iI''_i)< \frac12 \nu(K)$.   

\textbf{Step 2.} For any $A\subset U\subseteq [0,1]$ we claim that there exists  a finite family of disjoint intervals $\II=\{I_i\}\subset \JJ$ such that $I_i\subset U$ and $\nu(A\setminus\bigcup_iI_i)< \frac34 \nu(A)$.

In fact, there exist a compact  $K\subset A$ such that $\nu(A\setminus K) <\frac14 \nu(A)$. Appling step 1, we have a finite disjoint family of intervals such that
$$
\nu(A\setminus\bigcup_iI_i))\leq \nu(K\setminus\bigcup_iI_i))+ \nu(A\setminus K) \leq \frac12\nu(K)+\frac14 \nu(A)\leq \frac34 \nu(A).
$$
\textbf{Step 3.} We prove by induction that for any $n\in \N$  there exists  a finite family of disjoint intervals $\II^n:=\{I_i^n\}\subset \JJ$ such that $I_i^n\subset U$ and 
\begin{equation}\label{eq-vit-ind}
\nu(A\setminus\bigcup_iI_i^n)< \left(\frac34\right)^n \nu(A).
\end{equation}

Step 2 give a proof for $n=1$. Suppose that (\ref{eq-vit-ind}) holds for $n$. Take $A_n=A\setminus\bigcup_iI^n_i$ and $U_n=[0,1]\setminus\bigcup_iI^n_i$. Then, by step 2, there exists $\hat \JJ^n=\{J_i^n\}\subset \JJ$ such that $I_i^n\subset U_n$ (therefore the intervals from $\JJ^n$ are disjoint from the intervals  $\II^n$) and $\nu(A_n\setminus\bigcup_i J_i^n)< \frac34 \nu(A_n)$. 
Setting $\II^{n+1}=\JJ^n\cup \II^n$ we easily check (\ref{eq-vit-ind}) for $n+1$. 

Finally, set $ \II:=\cup_n\II^n$. We conclude the proof by taking the limit in (\ref{eq-vit-ind}) for $n\to \infty$.
\end{proof}
\begin{proof}[Proof of Proposition {\ref{prop-RN-deri-NC-int}}]
	(1). We refer to Federer's book   \cite{Federer1996}. Indeed,
Lemma \ref{lem-vitali} proves that $V=\{(x, I): I \mbox{ a closed interval and } x\in\mathrm{int}(I)\}$ are a $\nu$--Vitali relation according to \cite[Definition 2.8.16.]{Federer1996} 
	Let $$D(\lambda,\nu,x):=\lim_{\delta\to 0}\sup_{I\in \II_\delta^x}\frac{\lambda(I)}{\nu(I)}.
	$$
 By \cite[Theorem 2.9.7]{Federer1996}, $D(\lambda,\nu,x)d\nu(x)$ is an absolutely continuous part in the Lebesgue decomposition of $\lambda$, thus $D(\lambda,\nu,x)=\frac{\d\lambda}{\d\nu}(x)$.
	
	(2) Consider the set:
	$$
	A_a:=\left\{x\in [0,1]: Q^*(\lambda,\nu,x)>a\right\}.
	$$
	For every $x\in A_a$ choose a closed interval $I_x$ such that $x\in\mathrm{int}(I_x)$ and $\lambda(I_x)>a\mu(I_x)$. It is possible to extract from the family $\{I_x\}_{x\in A_a}$ a finite covering $\{I_i\}_i$ of $A_a$ such that every interval may overlaps with at most one other interval at each point,  that is
	$$
	\mathbf{1}_{A_a} \leq \sum_{i} \mathbf{1}_{I_i} \leq 2.
	$$    
	(This is Besicovitch's covering Theorem in dimension 1.) Then
	$$\nu(A_a)\leq \sum_{i} \nu(I_i)\leq\sum_{i} \lambda(I_i)/a\leq\int_{[0, 1]}\sum_{i} \mathbf{1}_{I_i}/a\,\d\lambda \leq 2\lambda([0,1])/a.
	$$
	 Now by the Fubini theorem we easily see that
\begin{align*}
		\int_{[0, 1]} \ln^+Q^*(\lambda,\nu,x)\nu(\d x)
		&=\int_{[0, 1]} \int_1^{\infty} \mathbf{1}_{[Q^*(\lambda,\nu,x)>a]} \frac{\d a}{a}\nu(\d x)= \int_1^{\infty}\nu(A_a)\frac{\d a}{a}\\
		&\leq  2\lambda([0,1]) \int_1^{\infty}\frac{\d a}{a^2} = 2\lambda([0,1]) .
\end{align*}
This completes the proof.
\end{proof}	

\subsection{Proof of exponential contraction}
\begin{myprop}{\ref{contractivity}}[Contractivity]
Let $\mu$ be a Borel probability measure on $\MM $, and let $\nu$ be a $\mu$--invariant atomless ergodic Borel probability measure on the interval $[0,1]$ such that $\ln^+ J$ is $\nu\times\mu$--integrable and  
	$h_{\mu}(\nu)>0$.
	Then for every ${h}\in (0, h_{\mu}(\nu))$ and
	$\mu^{\mathbb{N}}$-a.e. $\omega=(g_n)_{n\in\mathbb{N}}\in \MM^{\mathbb{N}}$ and $\nu$--a.e. $x\in [0,1]$ there exists a closed interval $I=I(\omega,x)$ with $x\in \mathrm{int}(I)$ such that
	$$
	|g_n\circ \ldots \circ g_1(I)|\leq \exp(-n\cdot {h}), \quad \forall n\in\mathbb{N}.
	$$
\end{myprop}
Define
 $$
 \begin{aligned}
&J_{\epsilon}(x,g)\\
&:=\sup\left\{\frac{\nu(g (I))}{\nu(I)}\, : \, I\subset [0, 1] \text{ closed interval}, x\in \mathrm{int}(I) \text{ and }\nu(I)\leq \epsilon \right\}.
\end{aligned}
$$
 \begin{lem} Suppose that $\nu$ is an atomless probability measure and $g\in\MM $. Then we have
 		\begin{equation}\label{eq-conv-J-eps}
 			\lim_{\varepsilon\to 0} J_{\epsilon}(x, g)=\d_{\nu} g(x)\qquad\mbox{$d\nu(x)$--a.s.}
 		\end{equation}  	
Furthermore,  if $\ln^+ J$ is $\nu\times \mu$--integrable, then  $$
h^{\epsilon}_{\mu}(\nu):=\int_{[0, 1]} \int_{\MM} \ln (J_{\epsilon}(x,g ))\,  \d\nu(x)\, \d\mu(g)
$$ 
is well defined, $h^{\epsilon}_{\mu}(\nu)\in (-\infty, +\infty]$, and 
\begin{equation}\label{eq-conv-h}
\lim_{\varepsilon\to 0} h^{\epsilon}_{\mu}(\nu)=h_{\mu}(\nu).
\end{equation} 

 \end{lem}

\begin{proof}  Denote $J_0:=\lim_{\varepsilon\to 0} J_{\epsilon}=\inf_{\varepsilon>0} J_{\epsilon}$. First observe that since $\nu$ has no atoms $\varepsilon(r,x):=\nu([x-r,x+r])\to 0$ as $r\to 0$. Thus
	$$
	\d_\nu g(x)=\lim_{r\to 0}\frac{\nu(g([x-r, x+r]))}{\nu([x-r, x+r])} \leq \lim_{r\to 0} J_{\varepsilon(r,x)}(x, g)=J_0(x,g).
	$$
	
To prove the reverse inequality observe that
$$Q^*(x):=\sup\left\{\frac{|I|}{\nu(I)} : \, I \text{ closed interval of [0,1] s.t. } x\in \mathrm{int}(I)  \right\}$$
coincide with $Q^*_\nu(\lambda,x)$ when $\lambda$ is the Lebesgue measure on $[0,1]$ and thus it is $\nu$--a.s. finite, by Proposition \ref{prop-RN-deri-NC-int} (2). 
Let $I_n^x$ be a sequence of intervals such that $J_0(x)=\lim_n \nu(g(I_n^x))/\nu(I_n^x)$  and such that $\nu(I_n^x)\to 0$. Then $\delta(x, n):=|I_n^x|\leq Q^*(x)\nu(I_n^x)\to 0$. Finally, since locally $g^{-1}\nu$ coincide with the measure $\nu_i^g:=(\gamma_i^{-1}\nu)\vert_{I_i^g}$ we can apply Proposition \ref{prop-RN-deri-NC-int} (1) to $\lambda=\nu_i^g$ and
$$
\begin{aligned}
J_0(x,g)&=\lim_{n\to \infty} \frac{\nu(g(I_n^x))}{\nu(I_n^x)}=\lim_{n\to \infty} \frac{\nu_i(I_n^x)}{\nu(I_n^x)}\\
&\leq \lim_{n\to \infty} \sup_{I\in\II_{\delta(x,n)}}\frac{\nu_i(I_n^x)}{\nu(I_n^x)}=\frac{d\nu_i}{d\nu}(x)=d_\nu g(x)
\end{aligned}
$$
This proves (\ref{eq-conv-J-eps}).

	To prove condition (\ref{eq-conv-h}) observe that $\ln^+J_{\epsilon}(x,g )$ converges $\nu\otimes\mu^{\otimes \mathbb N}$--a.s. to $\ln^+(d_{\nu}g(x))$, and since
	$\ln^+J_{\epsilon}(x,g )\le \ln^+J(x,g )$ and $\ln^+J(x,g )\in L^{1}([0, 1]\times\MM^{\mathbb N},\nu\otimes \mu^{\otimes \mathbb N})$ we have
	$$
	\lim_{\varepsilon\to 0}\int_{[0, 1]} \int_{\MM} \ln^+(J_{\epsilon}(x,g ))\,  \d\nu(x)\, \d\mu(g)=\int_{[0, 1]} \int_{\MM} \ln^+\,\d_{\nu} g(x)  \d\nu(x)\, \d\mu(g).
	$$
	On the other hand, since the sequence $\ln^-(J_{\epsilon}(x,g ))$ is decreasing as $\varepsilon\to 0$, by the Lebesgue theorem we also have
	$$
	\lim_{\varepsilon\to 0}\int_{[0, 1]} \int_{\MM} \ln^-(J_{\epsilon}(x,g ))\,  d\nu(x)\, \d\mu(g)=\int_{[0, 1]}\int_{\MM} \ln^-\,\d_{\nu} g(x)  \d\nu(x)\, \d\mu(g),
	$$
	and consequently (\ref{eq-conv-h}) holds.
\end{proof}
\begin{proof}[Proof of Proposition \ref{contractivity}] Consider the transformation $T:X\times \MM^{\mathbb{N} }\longrightarrow  X\times \MM^{\mathbb{N}}$  given by the formula
	$$
	T(x,\omega):=(g_1 (x), \sigma(\omega)),
	$$ 
	where $\sigma$ denotes the left shift on $\MM^{\mathbb N}$, i.e. $\sigma(g_1, g_2, \ldots)=(g_2, g_3, \ldots)$. 
	The transformation $T$ is $\nu\otimes \mu^{\mathbb{N}}$--measure preserving and ergodic for $\nu$ ergodic and $\mu$-invariant (see \cite{BQ}).  
	For any $\varepsilon>0$,  let $\tilde J_{\varepsilon}(x,\omega)= J_{\varepsilon}(x,g_1)$ and thus $ \tilde J_{\varepsilon}(T^k(x,\omega))=J_{\varepsilon}(X_k^x,g_{k+1})$.  
	
	Fix $\widetilde{h}\in (0, h_{\mu}(\nu))$ and choose $\varepsilon$ such that $\widetilde{h}<h^{\varepsilon}_{\mu}(\nu)$.
	
	By Birkhoff's Ergodic Theorem,
	$$\frac{\sum_{k=0}^{n-1}\ln (\tilde J_{\varepsilon}(T^k(x,\omega))}{n}
	\xrightarrow[n\to \infty]{}
	\int_{[0, 1]\times \MM} \ln (J_{\epsilon}(y,g))\, d(\nu\otimes \mu)(y,g) =-h_{\mu}^{\varepsilon}(\nu)
	$$ 
	for $\nu\otimes \mu^{\otimes \mathbb{N}}$--a.e. $(x,\omega)\in [0, 1]\times \MM^{\mathbb{N}}$.
	The above is still true even if $\ln\circ J_{\epsilon}\notin L^{1}([0, 1]\times \MM^{\mathbb{N}},\nu\otimes \mu^{\otimes \mathbb{N}})$, (thus $h_{\mu}^{\varepsilon}(\nu)=+\infty$) because we can consider the function $\max\{\ln\circ J_{\epsilon}, -M\}$ for arbitrary large $M>0$, which is in $L^{1}([0, 1]\times \MM^{\mathbb{N}},\nu\otimes \mu^{\otimes \mathbb{N}})$.

	 Then there exists $N\in\mathbb{N}$ such that 
	$$\widetilde{h}\leq -\frac{1}{n}\log\Big(\prod_{k=0}^{n-1}\tilde J_{\epsilon}(T^k(x,\omega))\Big)\quad\text{for all $n\geq N$}.
	$$ 
	This is equivalent to 
	$$ \prod_{k=0}^{n-1} J_{\epsilon}(X_k^x,g_{k+1})) =\prod_{k=0}^{n-1} \tilde J_{\epsilon}(T^k(x,\omega)) \leq e^{-n\widetilde{h}}  \quad\text{for all $n\geq N$}.
	$$
	Thus  for any $(x,\omega)$ be such that the Birkhoff limit holds there exists a constant $C=C(x,\omega)>0$ 
	such that 
	\begin{equation}\label{eq: prod} 
	 \prod_{k=0}^{n-1} J_{\epsilon}(X_k^x,g_{k+1}))  \leq C(x,\omega) e^{-n\widetilde{h}}
		\quad\text{for all $n\in\mathbb N$}.
	\end{equation}
    
    Let $$
    \Omega_1:=\left\{ (x,\omega): \mbox{ (\ref{eq: prod}) holds and } x\not\in \bigcup_{n\in\N} \DD(g_n\cdots g_1) \right\}.
    $$
    Since $\nu$ is atomless and  $ \bigcup_{n\in\N} \DD(g_n\circ \cdots \circ g_1) $ is countable, then $\Omega_1$ has full measure.
      
	Take $(x,\omega)\in \Omega_1$.  By induction we show that: for any closed $I_0$   containing $x$ in its interior such that $\nu(I_0)<\delta(x,\omega):=\frac{\varepsilon}{1+C(x,\omega)}$ we have
	\begin{equation}\label{eq: nu contracting} \nu(g_n\circ \cdots \circ g_1 (I_0))\leq C e^{-n\cdot \widetilde{h}}\nu(I_0)\quad\text{for all $n\in\mathbb N$} .
	\end{equation}
	 For $n=0$ the above is trivial. 
	
	Assume it is true for $k=0,\ldots,n-1$. Let $I_k:=g_k\circ \ldots \circ g_1 (I_0)$ Hence, by induction hypothesis
	$$\nu(I_k)=\nu(g_k\circ \ldots \circ g_1 (I_0))\leq C e^{-n\cdot \widetilde{h}}\nu(I_0)\leq C e^{-n\cdot \widetilde{h}}\delta \leq \varepsilon$$
	since $e^{-n\cdot \widetilde{h}}\leq 1$ because  $\widetilde{h}>0$. 

Observe that $X_{k}^x=g_k\circ \ldots \circ g_1(x)\in I_k$ since $x\in I_0$.
Furthermore, the fact that the functions $g_i$ are piecewise monotone and  $x\not\in D(g_{k+1}\circ \ldots \circ g_1)$, ensures also that $X_{k}^x\in\mathrm{int}(I_k)$.
	Therefore,
	$$
	\begin{aligned}
		&\frac{\nu(g_{k+1}\circ \ldots \circ g_1 (I_0))}{\nu(g_k\circ \ldots \circ g_1 (I_0))}=\frac{\nu(g_{k+1}(I_k))}{\nu(I_k)}\leq\\
		&\leq 
		\sup\left\{\frac{\nu(g_{k+1}(I))}{\nu(I)}\ :\  I \text{ an interval with }X_k^x\in \mathrm{int}(I),\, \nu(I)\leq \varepsilon\right\}\\
		&=
		J_{\epsilon}(X_k^x, g_{k+1})=\tilde J_{\varepsilon}(T^k(x,\omega)).
	\end{aligned}
	$$ 
	Now, applying the above inequalities, we obtain
	\begin{equation}\label{e1_3.01.24}
		\begin{aligned}
			\nu(g_n\circ \ldots \circ g_1 (I_0)) &= 
			\nu(I_0)\prod_{k=0}^{n-1} \frac{\nu(g_{k+1}\circ \ldots \circ g_1 (I_0))}{\nu(g_k\circ \ldots \circ g_1 (I_0))}\\
			&\leq 
			\nu(I_0)\prod_{k=0}^{n-1} J_{\epsilon}(T^k(x,\omega))
			\leq
			C \nu(I_0) e^{-n\widetilde{h}},
		\end{aligned}
	\end{equation}
	where the last inequality follows by (\ref{eq: prod}) and the proof of (\ref{eq: nu contracting}) is complete.
	
	Let now
	$$
	Q^*(x):=\sup\left\{\frac{|I|}{\nu(I)}: \text{$I$ a closed interval such that $x\in\mathrm{int}(I)$}\right\}.
	$$
	By Proposition \ref{prop-RN-deri-NC-int} (2)
	$\ln^+ Q\in L^{1}([0, 1],\nu)$. Therefore, by the Birkhoff Ergodic Theorem we obtain that
	$\frac{1}{n}\sum_{k=0}^{n-1} \ln^+ Q(X_k^x)$ converges 
	to $\int_{[0, 1]} \ln^+ Q \d\nu$ for $\nu\otimes \mu^{\otimes \mathbb{N}}$--almost every $(x, \omega)$. Hence
	$\ln^+ Q^*(X_n^x)/n$ tends to $0$ for $\nu\otimes \mu^{\otimes \mathbb{N}}$--almost every $(x, \omega)$. 
On the other hand, as observed above, $X_n^x\in\mathrm{int}(I_n)$, if $x\not\in \DD(g_{n}\circ\cdots\circ g_1)$, thus for $\nu$--a.e. $x$ (since $\nu$ is atomless). Hence for every $\kappa>0$ we obtain
	$$
	\frac{|g_{n}\circ\cdots\circ g_1(I_0)|}{\nu(g_{n}\circ\cdots\circ g_1(I_0))}=\frac{|I_n|}{\nu(I_n)}\le Q^*(X_n^x)\le e^{n\kappa}
	$$
	for $\nu\otimes \mu^{\otimes \mathbb{N}}$--almost every $(x, \omega)$ for large enough $n$.
	Consequently, from (\ref{e1_3.01.24}) it follows that for every $n\in\mathbb{N}$ we have
	$$
	|g_{n+1}\circ g_{n}\circ\cdots\circ g_1(I_0)|\le \widetilde{C}  e^{-n(\widetilde{h}-\kappa)}\nu(I_0)
	$$ for some constant $\widetilde{C}>0$.
	Since $\widetilde{h}$ was an arbitrary constant from $(0, h_{\mu}(\nu))$ and $\kappa>0$ is arbitrary small, decreasing $I_0$ if necessary, we complete the proof.
\end{proof}


\begin{thebibliography}{0}
	
	
	\bibitem[{AM}]{AM} L. Alsed\`{a} and M. Misiurewicz, \textit{Random interval homeomorphisms}, Proceedings of New Trends in Dynamical Systems. Salou, 2012, Publ. Mat., 15--36 (2014).
	
	\bibitem[{ABB}]{ABB}	
	\newblock G. Alsmeyer, S. Brofferio, and D. Buraczewski,
	\newblock \textit{Asymptotically linear iterated function systems on the real line},
	\newblock \textit{Ann. Appl. Probab.} \textbf{33} (2023), no. 1, 161--199.
	
	\bibitem[{AV}]{AV}
	A. Avila, M. Viana, \textit{Extremal {L}yapunov exponents: an invariance principle and applications},
	Inventiones Mathematicae, {181} (2010), no. 1, 115--178.
	
	\bibitem[{BBS}]{BBS}S. Brofferio, D. Buraczewski, and T. Szarek, \textit{On uniquness of invariant measures for random walks on $HOMEO^+(\mathbb{R})$.}, Ergodic Theory and Dynamical Systems (2021), 1--32. DOI 10.1017/etds.2021.31.
	
	\bibitem[{BQ}]{BQ}Y. Benoist, J.-F. Quint, \textit{Random walks on reductive groups.} Springer, Cham, 2016. 
	
	\bibitem[{C}]{C} D. Czapla, {\it On the Existence and Uniqueness of Stationary Distributions for Some Class of Piecewise Deterministic Markov Processes With State-Dependent Jump Intensity.} preprint,  arXiv:2303.11576v4 [math.PR] .
	
	\bibitem[{CS}]{CS} K. Czudek, T. Szarek, \textit{Ergodicity and central limit theorem for random interval homeomorphisms.} Israel J. Math. 239 (2020), no. 1, 75--98.
	
	\bibitem[{DKN}]{DKN} B. Deroin, V. Kleptsyn, and A. Navas, {\it Sur la dynamique unidimensionnelle en r\'egularit\'e interm\'ediaire.} Acta Math. {199} no. 2, 199-262 (2007). 

\bibitem[{Fed}]{Federer1996}
\newblock H. Federer,
\newblock {\it Geometric Measure Theory},
\newblock Springer, 1996.
	
	
	\bibitem[{GH}]{GH} M. Gharaei, A. J. Homburg, \textit{Random interval diffeomorphisms.} Discrete Contin. Dyn. Syst. Ser. S 10 (2017), no. 2, 241--272. 
	
	\bibitem[{HKRVZ}]{HKRVZ} A.J. Homburg, C. Kalle, M. Ruziboev, E. Verbitskiy, and B. Zeegers,
	\textit{Critical intermittency in random interval maps. }
	Comm. Math. Phys. 394 (2022), no. 1, 1--37. 
	
	\bibitem[{KL}]{KL}  S. Kullback, R. Leibler   \textit{On information and sufficiency}. Annals of Mathematical Statistics. 22  (1951), (1): 79--86. 
	
	
	\bibitem[{K}]{K} B. Kloeckner, {\it Optimal transportation and stationary measures for Iterated Function Systems},
	Math. Proc. of the Cambridge Philos. Soc. 173 (2022), no. 1. 163--187.
	
	\bibitem[{LM}]{LM} A. Lasota, M. C. Mackey, \textit{Chaos, Fractals, and Noise: Stochastic Aspects of Dynamics.} Applied Mathematical Sciences, vol. 97 (New York: Springer-Verlag), 1994. 
	
	\bibitem[{Le}]{Led86} F.Ledrappier, {\it Positivity of the exponent for stationary sequences of matrices},
	Lecture Notes in Math., {\bf 1186}, {56-73} (1986).
	
	
	
	\bibitem[{M}]{M} D. Malicet, \textit{Random Walks on $Homeo(S^1)$.} Comm. Math. Phys. 356 (2017), 1083--1116. 
	
	
	\bibitem[{Mat}]{Mat} P. Mattila,\textit{Geometry of Sets and Measures in Euclidean Spaces: Fractals and Rectifiability.} Cambridge University Press, 1995.
	
	\bibitem[{N}]{N} A. Navas, \textit{Groups of circle diffeomorphisms.} Chicago Lectures in Mathematics. University of Chicago Press, Chicago, IL, 2011. xviii+290 pp.
	
	\bibitem[{NvS}]{NvS} T. Nowicki, S. van Strien, \textit{Invariant measures exist under a summability condition for unimodal maps.} Invent. Math. 105 (1991), no. 1, 123--136.
\end{thebibliography}
 \end{document}